\documentclass{amsart}
\usepackage[utf8]{inputenc}
\usepackage{color}
\usepackage{amsmath,amsthm,amssymb,verbatim,mathrsfs}
\usepackage[colorlinks=false]{hyperref}

\title{A vertical Sato-Tate law for GL(4)}
\author{Tian An Wong}
\address{University of Michigan-Dearborn, Dearborn, 48128 MI, USA}
\email{tiananw@umich.edu}

\newtheorem{thm}{Theorem}[section]

\newtheorem{lem}[thm]{Lemma}
\newtheorem{prop}[thm]{Proposition}

\theoremstyle{remark}\newtheorem{rem}[thm]{Remark}

\def\GL{\ensuremath {\mathrm{GL}}}
\def\SL{\ensuremath {\mathrm{SL}}}
\def\Q{\ensuremath {{\mathbb Q}}}
\def\Z{\ensuremath {{\mathbb Z}}}
\def\C{\ensuremath {{\mathbb C}}}
\def\R{\ensuremath {{\mathbb R}}}
\def\re{\ensuremath {{\text{Re}}}}
\newcommand{\abs}[1]{\lvert #1 \rvert}

\begin{document}

\begin{abstract}
We establish a doubly-weighted vertical Sato-Tate law for GL(4) with explicit error terms. The main ingredient is an extension of the orthogonality relation for Maass cusp forms on GL(4) of Goldfeld, Stade, and Woodbury from spherical to general forms, and without their assumption of the Ramanujan conjecture for the error term.
\end{abstract}

\maketitle


\section{Introduction}

The Kuznetsov trace formula is a fundamental tool in the analytic theory of automorphic forms. The case of GL($n$) for $n=2,3$ is by now well-established, and we refer to \cite{blomer} for an excellenet survey.  Recently Goldfeld, Stade, and Woodbury established a Kuznetsov formula for $n=4$, and using it deduced an orthogonality relation for Maass cusp forms \cite{GSW}. The orthogonality relation was proved for $n=3$ by Zhou, and was used to prove a doubly-weighted vertical Sato-Tate with rate of converegence.   The goal of this note is to establish the same result for $n=4$.

As was already remarked in \cite{GSW}, the techniques to obtain the such an application is already well-established. This note confirms that expectation. In order to do so, one has to generalize the orthogonality relation established in \cite[Theorem 1.1.1]{GSW} that was proved only for spherical Maass cusp forms. We shall provide crude estimates here for  general Maass cusp forms.
Let $\{\phi_j\}$ be an orthogonal basis of even Maass cusp forms for SL$_4(\Z)$ tempered at infinity with associated Langlands parameters 
\[
\alpha^{(j)} = \big(\alpha^{(j)}_1, \alpha^{(j)}_2, \alpha^{(j)}_3,\alpha^{(j)}_4\big)\in 
  \left(i \mathbb R\right)^4,
\]
and
$L$-th Fourier coefficient $A_j(L)$ where
$L = (\ell_1, \ell_2, \ell_3)\in \mathbb Z^3$. Let $\mathcal L_j = L(1, \mathrm{Ad}, \phi_j)$. Our first result is then the following extension of \cite[Theorem 1.1.1]{GSW} to general Fourier coefficients. 

\begin{thm}
\label{orth}
Let $T\ge1$ and $R\ge 14$. Let $h_{T,R}$ be the test function defined in \eqref{test}, and define $\omega_j(T) = {h_{T, R}(\alpha^{(j)})}/{\mathcal L_j}$. Then as $T\to\infty$, we have
\begin{multline*}
\sum^\infty_{j=1}A_j(L)\overline{A_j(M)}\omega_j(T) = \delta_{L,M}\Big(c_1T^{9+8R} +  c_2 T^{8+8R} +  c_3 T^{7+8R}\Big) \\
+ {O}_{\epsilon,R}\bigg( (LM)^{\frac12+\epsilon}T^{6+8R+\epsilon}  + (LM)^{\frac12+\epsilon} T^{5+8R+\epsilon} + (\ell_1m_1)^{15/2}(\ell_2m_2)^{7}(\ell_3m_3)^{15/2} T^{4+8R+\epsilon} \bigg).
\end{multline*}
where $LM = \ell_1\ell_2\ell_3m_1m_2m_3$. Also, $c_1,c_2,c_3>0$ are constants that depend at most on $R$, $\epsilon>0$, and $\delta_{L,M}=1$ if $L=M$ and zero otherwise. 
\end{thm}

\noindent Note that \cite[Theorem 1.1.1]{GSW} is the case of $\ell_2=\ell_3=m_2=m_3=1$. It relies on the assumption of the Ramanujan Conjecture for $n = 2,3$ due to \cite[Theorem 7.0.7]{GSW}, but this assumption should be able to be removed at the cost of a slightly worse error term, depending on bounds towards Ramanujan, such as shown in \cite[Appendix]{KTF3} for GL(3). In our theorem we do not assume the Ramanujan Conjecture, instead we use weaker estimates based on work of Chandee and Li \cite{CL}, namely a Ramanujan on average bound.


As an application, the orthogonality relation implies the weighted vertical Sato-Tate law by a result of Zhou \cite{Zhou}. Denote by $dx$ the Sato-Tate measure by pushing forward the normalized Haar measure on SU($4)$ to $\hat T_0 /W$ where $\hat{T}_0$ is the diagonal torus and $W$ the Weyl group. Also, let $\hat{T}$ be the standard maximal torus of SL$_4(\C)$.

\begin{thm}
\label{ST}
Suppose $\phi_j$ corresponds to an irreducible unramified automorphic representation $\pi_j$ of $\mathrm{PGL}_4(\mathbb A_\Q)$ with Satake parameter $a_j(\pi_p)$ at $p$.  For any continuous test function $f$ in $\hat T/W$, we have the equality
\begin{equation}
\label{conv1}
\lim_{T\to\infty} \dfrac{\sum_{j=1}^\infty f(a_j(\pi_p))\omega_j(T)}{\sum_{j=1}^\infty \omega_j(T)}= \int_{\hat T_0/W}f(x) dx.
\end{equation}
If further $f$ is a monomial function as in \eqref{monom}, we have
\begin{align}
\label{conv2}
\dfrac{\sum_{j=1}^\infty f(a_j(\pi_p))\omega_j(T)}{\sum_{j=1}^\infty \omega_j(T)}= &\int_{\hat T_0/W}f(x) dx\\
& + O(P^{\frac12+\epsilon}(T^{6+8R+\epsilon}+T^{5+8R+\epsilon}) + P^{8} T^{4+8R+\epsilon}),\notag
\end{align}
with $P = p^{i_1+i'_1+ i_2+i'_2 + i_3+i'_3}$. 
\end{thm}

Similar results were established by Matz and Templier \cite{MT} and Jana \cite{Jana} for general $n$, but without explicit error terms, while more recently \cite{LNW} proved related results as an application of Ramanujan on average. 

\section{The general orthogonality relation}

\subsection{Notation}

We shall first recall some basic definitions and properties that we shall need. For $s\in\mathbb C$ with $ \re(s) > 5/2$, the L-function associated to  $\phi_j$ is given by
 \begin{align*}
 L(s, \phi_j)  = \sum_{m=1}^\infty A_j(m,1,1) m^{-s} = \prod_p \Bigg(1 - \frac{A_j(p,1,1)}{p^{s}} 
   + \frac{A_j(1,p,1)}{p^{2s}}
    - \frac{A_j(1,1,p)}{p^{3s}}
     +\frac{1}{p^{4s}}\Bigg)^{-1}.
     \end{align*}
and 
\[
L(s, \phi_j) = \prod_p L_p(s, \phi_j)
\]
where
$$
L_p(s,\phi_j)= \prod_{j=1}^4\left(1-{\alpha_j(p)\over p^s}\right)^{-1}= \sum_{m\ge0}{A_j(p^m,1,1)\over p^{ms}} .
$$
In the special case of $\SL_4(\mathbb Z)$, the Langlands parameters $(\alpha_1,\alpha_2,\alpha_3,\alpha_4)$ associated 
to  $s = \frac14 + \left(v_1,v_2, v_3\right)$ are given by
\[
 \alpha_1 = 3v_1+2v_2+v_3, \quad \alpha_2 = -v_1+2v_2+v_3, \quad \alpha_3 = -v_1-2v_2+v_3, \quad \alpha_4 = -v_1-2v_2-3v_3; 
\]
 where
 \[ 
 v_1 = \frac{\alpha_1 - \alpha_2}{4}, \quad v_2 = \frac{\alpha_2 - 
\alpha_3}{4}, \quad v_3 = \frac{\alpha_3 - \alpha_4}{4}. 
\]
Recall Theorems 9.3.11 and 10.8.1 of \cite{Goldfeld} for $n=4$, the Hecke relations
\[
A(m_1m_1', m_2m_2', m_3m_3') = A(m_1, m_2, m_3) A(m_1', m_2', m_3') 
\]
for $(m_1m_2m_3,m_1'm_2'm_3')=1$, and
\begin{equation}
\label{HR2}
A(m,1,1)A(m_1,m_2,m_3) = \sum_{\substack{c_1c_2c_3=m\\ c_i|m_i}}A\left(\frac{m_1c_3}{c_1},\frac{m_2c_1}{c_2},\frac{m_3c_2}{c_3}\right).
\end{equation}
Also $
A(m_1,m_2,m_3) = \overline{A(m_3,m_2,m_1)}.$

\subsection{Test functions}

Let $\alpha = (\alpha_1,\alpha_2,\alpha_3,\alpha_4)\in\C^4$ with $\alpha_1+\alpha_2+\alpha_3+\alpha_4 = 0$.  Let  $T > 1$ with $T\to\infty$ 
and $R\geq 14$ with $R$ fixed. We consider the test function
\[
p_{T,R}^\sharp(\alpha)  := e^{\frac{\alpha_1^2+\alpha_2^2+\alpha_3^2+\alpha_4^2}{2T^2}}\mathcal F_R(\alpha)
      \prod_{1\leq j \ne k \leq 4}
      \Gamma\left(\textstyle{\frac{2+R+\alpha_j - \alpha_k}{4}}\right),
\]
where
\[
 \mathcal F_R(\alpha)  = \left(\prod_{\sigma\in S_4}  \Big(1+\alpha_{\sigma(1)}-\alpha_{\sigma(2)}-\alpha_{\sigma(3)}+\alpha_{\sigma(4)}\Big) \right)^{\frac{R}{24}}.
\]
Given $2\leq n\leq 4$, we define
\begin{equation*}
p_{T,R}^{\sharp,(n)}(\alpha)  := 
    \begin{cases}
      e^{\frac{\alpha_1^2+\dots+\alpha_n^2}{2T^2}}
      \displaystyle{\prod_{1\leq j \ne k \leq n}}
      \Gamma\left(\textstyle{\frac{2+R+\alpha_j - \alpha_k}{4}}\right)
      & \mbox{ if }n=2,3,\\
      e^{\frac{\alpha_1^2+\dots+\alpha_n^2}{2T^2}} \mathcal{F}_R(\alpha)
      \displaystyle{\prod_{1\leq j \ne k \leq n}}
      \Gamma\left(\textstyle{\frac{2+R+\alpha_j - \alpha_k}{4}}\right)
      & \mbox{ if }n=4.
      \end{cases}
\end{equation*}
Suppose that $\phi$ is a Maass cusp form for $\GL(n)$ with Langlands parameter $\alpha(\phi):=\alpha=(\alpha_1,\ldots,\alpha_n)\in \C^n$.  Then we define the test function
\begin{equation}
\label{test}
 h_{T,R}^{(n)}(\phi) := \frac{ \big\lvert p_{T,R}^{\sharp,(n)}(\alpha) \big\rvert^2}{\prod\limits_{1\leq j\neq k\leq n} \Gamma\left(\frac{1+\alpha_j-\alpha_k}{2}\right)}.
\end{equation}
We shall often omit the superscript $(n)$ above when $n=4$.

\subsection{The orthogonality relation} 
The orthogonality conjecture \cite[Conjecture 1.0.2]{GSW} states that for some constant $0<\theta<1$, the relation
\begin{equation}
\label{conj}
\sum^\infty_{j=1}A_j(M)\overline{A_j(M')}\frac{h_{T, R}(\alpha^{(j)})}{\mathcal L_j} = \delta_{M,M'}\sum^\infty_{j=1}\frac{h_{T, R}(\alpha^{(j)})}{\mathcal L_j} + O_{M,M'}\left(\frac{h_{T, R}(\alpha^{(j)})}{\mathcal L_j}\right)^\theta
\end{equation}
holds.  In the case $n=4$, \cite{GSW} prove this for the test function $h_{T,R}(\alpha)$ above, which can be written as
\begin{align*} & e^{\frac{\alpha_1^2+\alpha_2^2+\alpha_3^2+\alpha_4^2}{T^2}} \prod\limits_{\substack{1\leq j, k \leq 4\\ j\neq k}}
      \frac{
      \Gamma\left(\textstyle{\frac{2+R+\alpha_j - \alpha_k}{4}}\right)^2 
      }{
      \Gamma\left( \frac{1+\alpha_j-\alpha_k}{2}  \right)}\prod\limits_{\sigma\in S_4}  \Big(1+\alpha_{\sigma(1)}-\alpha_{\sigma(2)}-\alpha_{\sigma(3)}+\alpha_{\sigma(4)}\Big)^{\frac{R}{12}}.
\end{align*}
We assume each Maass cusp form $\phi_j$ is normalized so that its first Fourier coefficient $A_j(1,1,1) = 1.$  Let $\ell, m \in \mathbb Z$  with $\ell m\ne 0.$  

\begin{thm}[{\cite[Theorem 1.1.1]{GSW}}]
With notation as above, we have for $T\to\infty$, 
\begin{multline*}
\sum\limits_{j=1}^\infty  A_j(\ell,1,1) \overline{A_j(m,1,1)} \frac{h_{T,R}\left(\alpha^{(j)}   \right)}{\mathcal L_j} 
=   \delta_{\ell,m}\Big(c_1T^{9+8R} +  c_2 T^{8+8R} +  c_3 T^{7+8R}\Big) \\
+ {O}_{\epsilon,R}\bigg( |\ell m|^{{\frac25}+\epsilon}   T^{6+8R+\epsilon}  + \lvert \ell m\rvert^{\frac{7}{32}+\epsilon} T^{5+8R+\epsilon} + \lvert \ell m\rvert^{\frac{15}{2}} T^{4+8R+\epsilon} \bigg),
\end{multline*}
where $\delta_{\ell,m}$ is the Kronecker symbol and  $c_1,c_2,c_3>0$ are absolute constants which depend at most on $R$. Note that $h_{T,R}$ is of size $T^{8R}$ on the relevant support. The error term assumes the Ramanujan conjecture for $\mathrm{GL}(n)$ for $n\le 3$.
\end{thm}

This orthogonality relation is obtained by evaluating the Kuznetsov trace formula for GL(4). The cuspidal contribution to the Kuznetsov formula is given by
\[
\mathcal C =\sum_{j=1}^\infty  \frac{A_j(L)\overline{A_j(M)}\left|  p_{T,R}^\#\left( \overline{\alpha^{(j)}}\right)\right|^2}{ \mathcal L_j \prod\limits_{1\le j \ne k\le 4} \Gamma\left( \frac{1+\alpha_j-\alpha_k}{2}  \right)}.
\]
The main theorem of \cite{GSW} shows that for $N=4$, $M=(m,1,1)$, and $L=(\ell,1,1)$, that $\mathcal C$ is equal to the sum of
\begin{equation}
\label{M}
\mathcal M  =  \delta_{M,L}\left(c_1T^{9+8R}+c_2 T^{8+8R}+c_3T^{7+8R} + O(T^{6+8R}) \right) 
\end{equation}
and
\begin{equation}
\label{KE}
\mathcal K - \mathcal E  = {O}_{\epsilon,R}\bigg( |\ell m|^{{\frac25}+\epsilon}   T^{6+8R+\epsilon}  + \lvert \ell m\rvert^{\frac{7}{32}+\epsilon} T^{5+8R+\epsilon} + \lvert \ell m\rvert^{\frac{15}{2}} T^{4+8R+\epsilon} \bigg)
\end{equation}
where $\mathcal M, \mathcal K, \mathcal E$ are referred to as the main term, Kloosterman contribution, and Eisenstein contribution respectively.

We note that the estimates on $\mathcal M$ and $\mathcal K$ are obtained for general $M,L$ by Proposition 3.5.1 and Proposition 4.0.4 of \cite{GSW}. Therefore only the $\mathcal E$ estimate needs to be generalized to arbitrary $M,L$. In \cite[Remark 1.1.3]{GSW} the authors note that it is possible, using the Hecke relations, to obtain a more general version of Theorem but the formulas get quite complex and messy. We shall use the Hecke relations to get the more general estimate parallel to \cite[Theorem 7.0.7]{GSW}. The main term $\mathcal M$ being as above, we shall now describe $\mathcal K$ and $\mathcal E$ more explicitly.

\subsection{The Kloosterman contribution}

Let $M=(m_1,m_2,m_3),   L=(\ell_1,\ell_2,\ell_3) \in \mathbb Z^3$. Denote by $W_4\simeq S_4$ the Weyl group of GL$_4(\R)$. The Kloosterman contribution to the Kuznetsov trace formula is given by
\[
\mathcal K  =  C_{L,M}^{-1} \sum\limits_{w\in W_4} \mathcal I_w 
\] 
where 
\[
C_{L,M}=  c_4  (\ell_1 m_1)^3(\ell_2m_2)^4(\ell_3m_3)^3,
\]
 and
$\mathcal I_w$ is a sum of Kloosterman integrals defined in \cite[(4.0.1)]{GSW} and $c_4$ is a positive absolute constant. Let $r\geq 1$ be an integer.  Then for $R$ sufficiently large and any $\epsilon>0$, we have from \cite[Proposition 4.0.4]{GSW} that the Kloosterman contribution is bounded by
 \[ C_{L,M}^{-1}\big| \mathcal{I}_{w_j} \big| \ll_{\epsilon, R} (\ell_1m_1)^{2r-1/2}(\ell_2m_2)^{2r-1}(\ell_3m_3)^{2r-1/2} B_j(T), \]
where
\[ B_j(T) = 
 \begin{cases}
 T^{\epsilon+8R+20-4r} & \mbox{ if }j=2,3,4, \\
 T^{\epsilon+8R+19-5r} & \mbox{ if }j=6,7, \\
 T^{\epsilon+8R+18-6r} & \mbox{ if }j=5,8.
 \end{cases}\]
The index $j=1,\dots,8$ runs over elements in $W_4$. Optimising at $r=4$, we have that 
\[
\mathcal{K} = O_{\epsilon,R}((\ell_1m_1)^{15/2}(\ell_2m_2)^{7}(\ell_3m_3)^{15/2} T^{4 + 8R+\epsilon}). 
\]

\subsection{The Eisenstein contribution}

Suppose that $4=n_1+\dots+n_r$ is a partition of $4$ and $\Phi=(\phi_1,\ldots,\phi_r)$ where, for $1\le j\le r$, $\phi_j$ is  a Maass cusp form for $\SL(n_j, \mathbb Z)$ if $n_j>1$, while $\phi_j$ is the constant function (properly normalized) if $n_j=1$.   Let $\mathcal{P}=\mathcal{P}_{n_1,\ldots,n_r}$.  Then we define
\begin{equation*}
\mathcal E_{\mathcal P,\Phi} =  \int\limits_{\re(s_1)=0}\dots \int\limits_{\re(s_{r-1})=0}  A_{E_{\mathcal P,\Phi}}(L, s)
 \overline{A_{E_{\mathcal P,\Phi}}(M, s)} \left| p_{T,R}^\#\big(\alpha_{_{\mathcal P, \Phi}}(s)\big)\right|^2  ds_1 s ds_{r-1},
\end{equation*}
and denote
\begin{equation*}
\mathcal  E_{\mathcal P_{\rm\text{min}}} = \mathcal E_{\mathcal P_{1,1,1,1},\Phi},
\end{equation*}
where $\alpha_{_{\mathcal P, \Phi}}(s)$,  $\alpha_{_{\mathcal P_{\rm Min}}}(s)$ denote the Langlands parameters of the Eisenstein series $E_{\mathcal P,\Phi}(g,s)$,
$E_{\mathcal P_{\rm Min}}(g,s)$, respectively. Also, $A_{E_{\mathcal P, \Phi}}(L, s)$, $A_{E_{\mathcal P, \Phi}}(M, s)$ denote the $L$-th and $M$-th Fourier coefficients of  $E_{\mathcal P,\Phi}(g,s)$, and similarly for $E_{\mathcal P_{\rm Min}}(g,s)$. The following estimate follows from a Ramanujan on average bound of Chandee and Li.

\begin{lem}
\label{CLbound}
For positive integers $k,l,n$ and $\epsilon>0$, we have
\[
|A(k,l,n)|\ll (kln)^{\frac12+\epsilon}.
\]
\end{lem}
\begin{proof}
This follows from \cite[Lemma 3.4]{CL}, namely, the formula
\[
A(k,l,n)  = \sum_{d|(k,l)}\sum_{e|(d,k/d)}\sum_{f |(k,n)}\mu(d)\mu(e)A(\frac{k}{def},1,1)A(1,\frac{l}{d},\frac{dn}{ef}).
\]
and an application of \cite[Lemma 3.5]{CL}.
\end{proof}

Now define
\[
 \mathcal{E}_{\mathcal P}  = \sum_\Phi c_{L,M,\mathcal{P}}  \mathcal{E}_{\mathcal{P},\Phi}.
\]
Then the contribution to the Kuznetsov trace formula coming from the Eisenstein series is given by
\[
\mathcal{E}= c_1 \mathcal{E}_{\mathcal P_{\rm\text{min}}}+c_2\mathcal{E}_{\mathcal P_{2,1,1}}+c_3\mathcal{E}_{\mathcal P_{2,2}}+c_4\mathcal{E}_{\mathcal P_{3,1}},
\]  
 for constants $c_1,c_2,c_3,c_4 > 0$. We shall be satisfied with the ineffective estimate below.

\begin{prop}
\label{thm:EisensteinBound}
Let $L=(\ell_1,\ell_2,\ell_3), M= (m_1,m_2,m_3)$, and set $LM = \ell_1\ell_2\ell_3m_1m_2m_3$.  Then 
\begin{align*} 
& \abs{\mathcal{E}_{\mathcal{P}_{\rm\text{min}}}}  \ll_{\epsilon} (LM)^{\frac12+\epsilon} T^{3+8R+\epsilon}, 
&&\abs{\mathcal{E}_{\mathcal P_{2,1,1}}}  \ll_{\epsilon} (LM)^{\frac12+\epsilon} T^{2+8R+\epsilon}, \\
&\abs{\mathcal{E}_{\mathcal P_{2,2}}}   \ll_{\epsilon}  (LM)^{\frac12+\epsilon}   T^{5+8R+\epsilon}, 
&& \abs{\mathcal{E}_{\mathcal P_{3,1}}}   \ll_{\epsilon} (LM)^{\frac12+\epsilon}   T^{6+8R+\epsilon} ,  
\end{align*}
as $T\to\infty$ for any fixed $\epsilon>0$.
\end{prop}

\begin{proof}
The crucial step is to bound the product of Fourier coefficients, and the remainder of the analysis will follow in the same manner from the proof of \cite[Theorem 7.0.7]{GSW}. That is, we bound 
\[
A_{E_{\mathcal P,\Phi}}(L, s) \overline{A_{E_{\mathcal P,\Phi}}(M, s)}
\]
and
\[
A_{E_{\mathcal P_{\rm \text{\rm Min}}}}(L, s) \overline{A_{E_{\mathcal P_{\rm \text{\rm Min}}}}(M, s)}.
\]
using Lemma \ref{CLbound}.
\end{proof}

\begin{rem}
We note that our estimates here are much weaker, even in the spherical case, than that of \cite[Theorem 7.0.7]{GSW}. The reason for this is that the general Fourier coefficients require a calculation of the Hecke eigenvalues for general Hecke operators, which is not known, whereas the spherical case has been long known due to work of Goldfeld.
\end{rem}

Then putting the estimates for $\mathcal M,\mathcal K$, and $\mathcal E$ together in \eqref{M} and \eqref{KE} we obtain Theorem \ref{orth}.

\section{Sato-Tate with rate of convergence}

\subsection{Preliminaries}

We follow the setup of \cite{Zhou}. We first recall some facts about the representation theory of GL$(n)$. We specialize to $n=4$ for concreteness. Let $\pi = \otimes_{p\le \infty} \pi_p$ be a cuspidal automorphic representation of PGL$_4(\mathbb{A})$. For each finite prime $p<\infty$, we take $\pi_p$ to be an unramified admissible representation with Satake parameter
\[
a(\pi_p) = \begin{pmatrix}\alpha_{p,1} & & \\ & \ddots & \\ & & \alpha_{p,4}\end{pmatrix} \in \hat{T}/W.
\]
where $\hat T$ is the standard maximal torus of SL$_4(\mathbb C)$, and $W \simeq S_4$ is the usual Weyl group. The generalized Ramanujan conjecture implies that $|\alpha_{p,i}|=1$ for all $1 \le i \le 4$ and $p$. It is equivalent to the statement that $a(\pi_p)$ lies in the subtorus $\hat T_0/W$ of SU$(4)$. 

Let $\epsilon_i$ be the $i$-th standard basis vector in $\mathbb{R}^{4}$. Denote the root system of GL$_4(\mathbb C)$ by $\Phi=\{\epsilon_i -\epsilon_j: i\neq j\},$ and the set of positive roots $\Phi=\{\epsilon_i -\epsilon_j: i< j\}$. The set of integral weights $\Lambda$ of GL$_4(\mathbb C)$ is generated by the set of 
\[
\epsilon_i - \frac{1}{4}\sum_{j=1}^n\epsilon_j, \quad 1\le i \le 3.
\]
The Weyl chamber associated to $\Phi^+$ is then given by
\[
C = \left\{\sum_{i=1}^4a_i\epsilon_i: a_1\ge \dots \ge a_4, a_i\in\mathbb{R}, \sum_{i=1}^4a_i=0\right\}.
\]
Then given a weight $\mu$ in $\Lambda\cap C$, let $V_\mu$ be the highest weight representation of $\mu$ and $\chi_\mu$ its character. Formally, we may write $\chi_\mu$ as a finite sum of $e^\lambda$ for $\lambda\in\Lambda$ with nonnegative integer coefficients, invariant under the action of $W$. Let $V$ be the representation given by the standard emebedding of SL$_4(\mathbb{C})$ in GL$_4(\mathbb{C})$, and let $V_k = \wedge^k V$ be the $k$-th exterior product for $k=1,\dots,3$, so that $V_k$ corresponds to the highest weight representation of 
\[
\sum_{i=1}^k\Big(\epsilon_i - \frac{1}{4}\sum_{j=1}^n\epsilon_j\Big).
\]
Its character is given by
\[
\chi_k\Big(\begin{pmatrix}\alpha_{1} & & \\ & \ddots & \\ & & \alpha_{4}\end{pmatrix}\Big) =  \sum_{i_1<\dots<i_k}\alpha_{i_1}\dots  \alpha_{i_k},
\]
which is an elementary symmetric polynomial on $\hat T/W$. We define a map from $\hat T/W\to \C^{3}$ by
\[
\rho: \text{diag}(\alpha_1,\dots,\alpha_4) \to (\chi_1(\text{diag}(\alpha_1,\dots,\alpha_4)),\dots,\chi_n(\text{diag}(\alpha_1,\dots,\alpha_4))).
\]
And define a monomial function $f(\text{diag}(\alpha_1,\dots,\alpha_4))$ of the form
\begin{equation}
\label{monom}
\left(\sum_i \alpha_j\right)^{i_1}\left(\sum_i \bar \alpha_i\right)^{i_1'} \left(\sum_{i<j} \alpha_i\alpha_j\right)^{i_2}\left(\sum_{i<j} \overline{\alpha_i\alpha_j}\right)^{i_2'}\left(\sum_{i<j<k} \alpha_i\alpha_j\alpha_k\right)^{i_3}\left(\sum_{i<j<k} \overline{\alpha_i\alpha_j\alpha_k}\right)^{i_3'}
\end{equation}
for nonnegative integers $i_1,i_1',i_2,i_2',i_3,i_3'$. Define a bijection between $\Lambda\cap C$ and the set $\Z^{3}_{\ge0}$ of integer nonnegative triples
\[
\omega:(l_1,l_2,l_{3})\mapsto \sum_{i=1}^{3}\left(\sum_{k=1}^{3-i}l_k\right)\Big(\epsilon_i - \frac{1}{4}\sum_{j=1}^3\epsilon_j\Big).
\]
Then by the Casselman-Shalika formula we have that the Fourier coefficients of $\pi$ are given by $
A(p^{l_1},\dots,p^{l_k}) = \chi_{\omega(l_1,\dots,l_{n-1})}(a(\pi_p))$.

\subsection{Proof of the Theorem \ref{ST}}

The identity \eqref{conv1} follows immediately from \cite[Theorem 8.4]{Zhou}.  The proof of second \eqref{conv2} is essentially the same as that of \cite[Theorem 9.1]{Zhou}, as an application of the orthogonality relation and the Casselman-Shalika formula. We sketch the proof here.

Let $A_j(M)$ be the $M$-th Fourier coefficient of $\phi_j$. Denote $A_j(1,\dots,p,\dots,1)$ with $p$ in the $(4-k)$-th position as $A_j[k]$. It follows from the proof of \cite[Lemma 8.3]{Zhou} that
\begin{align*}
 \prod_{k=1}^{3}A_j[k]^{i_k}\overline{A_j[k]}^{i_k' }
&= \prod_{k=1}^{3}A_j[k]^{i_k}{A_j[4-k]}^{i_k' } \\
&= \prod_{k=1}^{3}\chi_k(a_k(\pi_p))^{i_k}\chi_{4-k}(a_k(\pi_p))^{i_k'}\\
& = \sum_{\mu\in\Lambda\cap C}a_\mu\chi_\mu(a_j(\pi_p)).
\end{align*}
The last line uses the property that 
\[
\prod_{k=1}^{3}\chi_k^{i_k}\chi_{3-k}^{i'_k} = \sum_{\mu\in\Lambda\cap C}a_\mu\chi_\mu
\]
where $a_\mu$ is the multiplicity of $V_\mu$ in the decomposition of the finite-dimensional representation
\[
\bigotimes^3_{k=1}(V_k^{\otimes i_k}\otimes V_k^{\otimes i'_k}) = \bigoplus_{\mu\in\Lambda\cap C}V_\mu^{\oplus a_\mu}.
\]
Then the assumption that $f$ is monomial translates to $f\circ\rho$ being given by
\[
(z_1,z_2,z_3) \mapsto z_1^{i_1}\bar z_1^{i_1'}z_2^{i_2}\bar z_2^{i_2'}z_3^{i_3}\bar z_3^{i_3'}, \qquad i_1,\dots,i_3'\in \Z_{\ge0},
\]
which implies by the analogous computation as in \cite[p.423]{Zhou} that 
\begin{align}
\label{diff}&\dfrac{\sum_{j=1}^\infty f(X_j(p))\omega_j(T)}{\sum_{j=1}^\infty \omega_j(T)} -  \int_{\hat T_0/W}f(x) dx \\
&= \sum_{l_1,l_2,l_3\ge0}a_{\omega((l_1,l_2,l_3))}\left(\frac{A_j(p^{l_1},p^{l_2},p^{l_3})\omega_j(T)}{\sum_j\omega_j(T)} - \delta_{l_1,0}\delta_{l_3,0}\delta_{l_3,0}\right),\notag
\end{align}
and applying the orthogonality relation of Theorem \ref{orth}, we may bound the inner expression by
\[
{O}_{\epsilon,R}\bigg( (p^{l_1+l_2+l_3})^{\frac12+\epsilon}T^{5+8R+\epsilon}(T+1) + p^{\frac{15 l_1}{2}+7l_2 +\frac{15l_3}{2}} T^{4+8R+\epsilon} \bigg).
\]
Also, by \cite[p.424]{Zhou} we have that for any $\alpha>0$, 
\[
\sum_{l_1,l_2,l_3\ge0}a_{\omega((l_1,l_2,l_3))}p^{\alpha(l_1+l_2,l_3)}  = O(p^{\alpha(i_1+i'_1+ \dots + i_3+i'_3)}),
\]
and it follows that \eqref{diff} equals 
\begin{align*}
&{O}\bigg(\sum_{l_1,l_2,l_3\ge0}a_{\omega((l_1,l_2,l_3))}( (p^{l_1+l_2+l_3})^{\frac12+\epsilon}(T^{6+8R+\epsilon}+T^{5+8R+\epsilon}) + p^{8(l_1+l_2+l_3)} T^{4+8R+\epsilon} \bigg)\\
&= O(P^{\frac12+\epsilon}(T^{6+8R+\epsilon}+T^{5+8R+\epsilon}) + P^{8} T^{4+8R+\epsilon}),
\end{align*}
with $P = p^{i_1+i'_1+ \dots + i_3+i'_3}$ as desired.

\subsection*{Acknowledgments} The author thanks Jack Buttcane and Keshav Aggarwal for helpful conversations related to this work. The author was partially supported by NSF grant DMS-2212924.

\bibliography{main}
\bibliographystyle{alpha}
\end{document}